\RequirePackage[l2tabu, orthodox]{nag}

\documentclass[12pt,reqno]{amsart}
\usepackage{fullpage,url,amssymb,enumerate,colonequals}
\usepackage{mathrsfs} 
\usepackage{enumitem}  
\usepackage{appendix}
\usepackage{multicol}
\usepackage{afterpage}
\usepackage{multirow}
\usepackage{pdflscape}
\usepackage{afterpage}
\usepackage{lscape}
\usepackage{tabularx}

\usepackage[dvipsnames,xcdraw,hyperref]{xcolor}

\newcommand{\m}[2]{\multirow{#1}{*}{#2}}

\newcommand{\C}{\mathbb{C}}

\newcommand{\Q}{\mathbb{Q}}

\newcommand{\R}{\mathbb{R}}
\newcommand{\Z}{\mathbb{Z}}

\newcommand{\Qbar}{{\overline{\Q}}}

\newcommand{\pp}{\mathfrak{p}}

\newcommand{\calD}{\mathcal{D}}

\newcommand{\calO}{\mathcal{O}}

\DeclareMathOperator{\Norm}{Norm}

\newcommand{\injects}{\hookrightarrow}
\newcommand{\intersect}{\cap} 
\newcommand{\isom}{\simeq}

\newcommand{\tensor}{\otimes} 

\newcommand{\union}{\cup} 

\newtheorem{theorem}{Theorem}[section]
\newtheorem{lemma}[theorem]{Lemma}
\newtheorem{corollary}[theorem]{Corollary}
\newtheorem{proposition}[theorem]{Proposition}

\theoremstyle{definition}

\theoremstyle{remark}
\newtheorem{remark}[theorem]{Remark}

\usepackage{scalerel,stackengine}
\stackMath
\newcommand\reallywidehat[1]{%
\savestack{\tmpbox}{\stretchto{%
  \scaleto{%
    \scalerel*[\widthof{\ensuremath{#1}}]{\kern.1pt\mathchar"0362\kern.1pt}%
    {\rule{0ex}{\textheight}}
  }{\textheight}%
}{2.4ex}}%
\stackon[-8.5pt]{#1}{\tmpbox}%
}

\makeatletter
\g@addto@macro\bfseries{\boldmath} 
\makeatother

\usepackage{tikz}

\usepackage[
	pagebackref,
	pdfauthor={}, 
	pdftitle={Dehn Invariant Zero Tetrahedra},
]{hyperref}

\usepackage{microtype}  

\makeatletter
\@namedef{subjclassname@2020}{%
  \textup{2020} Mathematics Subject Classification}
\makeatother

\begin{document}

\title{Dehn Invariant Zero Tetrahedra}

\author{A. Anas Chentouf}
\address{Department of Mathematics, Massachusetts Institute of Technology, Cambridge, MA 02139-4307, USA}
\email{chentouf@mit.edu}

\author{Yihang Sun}
\address{Department of Mathematics, Massachusetts Institute of Technology, Cambridge, MA 02139-4307, USA}
\email{kimisun@mit.edu}


\thanks{This research was supported in part by National Science Foundation grant DMS-1601946.}

\date{December 2, 2023}

\begin{abstract}
We survey literature on all known families and examples of Dehn invariant zero tetrahedra. We also contribute two previously unknown families of Dehn invariant zero tetrahedra. Following a suggestion of Dill and Habegger, we show that there are finitely many Dehn invariant zero tetrahedra whose dihedral angles have a five-dimensional span over $\mathbb{Q}$, a first step towards classifying all Dehn invariant zero tetrahedra.

\end{abstract}

\maketitle


\section{Introduction}
\subsection{History}
The Dehn invariant of a polyhedra (see Section~\ref{sec:notations}) was introduced by Max Dehn in 1901 to prove Hilbert's Third Problem that not all polyhedra with equal volume could be dissected into each other \cite{dehn1901ueber}. It is invariant under \emph{scissors-cutting}, where a polyhedron $P$ is sliced into smaller polyhedra using planes and these slices are rearranged to form polyhedron $P'$. In this case, the two polyhedra $P$ and $P'$ are said to be \emph{scissors-congruent}. 
In 1901, Dehn proved that two scissors-congruent polyhedra have equal volume and Dehn invariant \cite{dehn1901ueber}. The converse was shown by Sydler in 1965 \cite{MR0192407}. 

The Dehn invariant is also related to the study of space tiling polyhedra, which dates back to Aristotle \cite{doi:10.1080/0025570X.1981.11976933}. A polyhedron \emph{tiles} if one can place congruent copies of it in $\mathbb{R}^3$ so that every point lies in the interior of exactly one copy or at the intersection of boundaries of two or more copies. Debrunner showed that tiling polyhedra have Dehn invariant zero \cite{MR604258}.

The results on scissors-congruence and tiling motivated further work on cataloging Dehn invariant zero tetrahedra. In $1896$, Hill discovered infinite families of tetrahedra which tile space \cite{MR1576480}. In the $1970$s, Boltyanskii compiled a list of Dehn invariant zero tetrahedra and computed several properties \cite{boltianskiui1979hilbert}. This list includes the Hill families and those discovered by Goldberg \cite{goldberg1974three} and others.  
A classification of tetrahedra with rational dihedral angles (a subset of Dehn invariant zero tetrahedra) was conjectured by Poonen and Rubinstein in $1995$ and proven in $2020$ by the two authors with Kedlaya and Kolpakov \cite{kedlaya2020space}.

\subsection{Main results} 
Our first result is a new approach to classifying Dehn invariant zero tetrahedra. The $\Q$-span $V\subset \R/\Q\pi$ of the six dihedral angles of a Dehn invariant zero tetrehedron has dimension between $0$ and $6$. We divide into cases numbered by this dimension.
\begin{itemize}
\item Case $0$: this requires all the dihedral angles to lie in $\Q\pi$; \cite{kedlaya2020space} classifies them.
\item Case $2$: we provide in Theorem~\ref{new-fam} two new one-parameter families of tetrahedra.
\item Case $5$: we prove Theorem~\ref{main-case5} in Section~\ref{case5} that there are finitely many such tetrahedra.
\item Case $6$: by linear independence, the vanishing of the Dehn invariant implies that all the edge lengths $e_{ij}$ are $0$; there are no such nondegenerate tetrahedra.
\end{itemize}

\begin{theorem}\label{main-case5}
Any Dehn invariant zero tetrahedron with a five dimensional $\Q$-span in $\R/\Q\pi$ of the dihedral angles is similar to one with integer side lengths, each at most $3.946 \times 10^{12}$.
\end{theorem}

To classify Dehn invariant zero tetrahedra, it is useful to consider some generic symmetries of tetrahedra: there is the $S_4$ symmetry from permuting the vertex labels, as well as Regge symmetries spelled out in Theorem~\ref{regge} from \cite{akopyan2019regge}. The two families found in Case 2 would be new in the sense of both of these symmetries. We summarize our results there as follows.

\begin{theorem}\label{new-fam}    
With real number parameter $t > 4$, consider the following two one-parameter families of tetrehedra labeled by the six edge lengths in the order $(12, 34, 13, 24, 14, 23)$: 
{\small \begin{equation}\label{E:TT'}
T=\left(t+1, t-1, t+1, t-1, t, 6\right)\quad \text{and}\quad T'=\left(t+1, t-1,\frac{t}{2}+2, \frac{t}{2}+4, \frac{t}{2}+3, \frac{3t}{2}-3\right).
\end{equation}}
\begin{enumerate}
    \item Both families have Dehn invariant zero.
    \item The dimensions of the spans of the dihedral angles over $\Q$ are both generally $2$.
    \item Both families are ``new'' (i.e. they are not identical or Regge equivalent to any of the three Hill families) and they are the only families in their Regge equivalence class.
\end{enumerate}
\end{theorem}
Finally, we compile a list of known Dehn invariant zero tetrahedra. Table 1 lists the one-parameter families up to $S_4$ and Regge symmetries; Table 2 lists just $S_4$ orbits. We also compute dimension of the $\Q$-span in $\R/\Q\pi$ and contribute some isolated examples in Table 3.

\subsection{Notation}\label{sec:notations}
We use the tetrahedra notations in \cite{wirth2014relations,kedlaya2020space}. We label the vertices of a tetrahedron as $\lbrace 1, 2, 3, 4\rbrace $.
Let $e_{ij}$ be the length of the edge $ij$, and let $\alpha_{ij}$ be the dihedral angle along that edge.
We always list dihedral angles and edge lengths  with subscripts in the order $(12, 34, 13, 24, 14, 23)$. The \emph{Dehn invariant} of a tetrahedron is given by
\[
     \calD := \sum_{1\le i<j\le 4} e_{ij} \tensor \alpha_{ij} \; \in \; \R \tensor_{\Q} (\R/\Q\pi).
\]
Let $\textrm{i}=\sqrt{-1}$ and reserve the italicized $i$ for indexing tetrahedra. Let $\Q_p$ be the completion of $\Q$ with respect to the $p$-adic absolute value. Let $\Qbar_p$ be an algebraic closure of $\Q_p$.
\subsection*{Acknowledgements}
The authors would like to extend our profound gratitude towards Professor Bjorn Poonen for advising this project and for many insightful discussions. The authors would like to thank Gabriel Dill and Philipp Habegger for suggesting the method for Theorem~\ref{main-case5}, as well as Cerine Hamida and Shahzod Nazirov for help in translating Russian. 

\section{A \texorpdfstring{$p$}{p}-adic Approach to Case 5}
\label{case5}
In this section, we will first apply the method of $p$-adic valuations and combine it with the new approach of height functions, thereby proving Theorem~\ref{main-case5}.
\subsection{Setting up the valuations}\label{sec-val}
\begin{lemma}\label{lem:case5-integral}
In Case 5, $\{ e_{ij}\}_{ij}$ can be scaled to be simultaneously integral 
and $\sum e_{ij} \alpha_{ij} 
\in \Q\pi$.
\end{lemma}
\begin{proof}
By assumption, the map $\Q^6 \to \R/\Q \pi$ with $(q_{ij})\mapsto \sum_{i<j} q_{ij} \alpha_{ij}$ has a one-dimensional kernel $\Q \vec{a}$, for some $\vec{a} \in \Q^6$; we may scale $\vec{a}$ to assume that $\vec{a} \in \Z^6$ and $a_{ij}>0$ for some $i,j$.  Tensoring with $\R$ gives a map $\R^6 \to \R \tensor_{\Q} \R/\Q\pi$ with kernel $\R \vec{a}$. Now, $\vec{e} = (e_{ij}) \in \R^6$ maps to $\calD =0$, so $\vec{e} \in \R \vec{a}$. After scaling, $\vec{e} = \vec{a} \in \Z^6$, and $\sum e_{ij} \alpha_{ij} = \sum a_{ij} \alpha_{ij} = 0$ in $\R/\Q \pi$.
\end{proof}
\begin{proposition}[{\cite[Lemma 5 and Theorem 1]{wirth2014relations}}]\label{prop-WD}
Define the Cayley-Menger matrix 
\[ M := \begin{pmatrix}
0 & e_{12}^2 & e_{13}^2 & e_{14}^2 & 1 \\
e_{12}^2 & 0 & e_{23}^2 & e_{24}^2 & 1 \\
e_{13}^2 & e_{23}^2 & 0 & e_{34}^2 & 1 \\
e_{14}^2 & e_{24}^2 & e_{34}^2 & 0 & 1 \\
1 & 1 & 1 & 1 & 0
\end{pmatrix}.\]
Let $D:=\det M$, let $D_{ijk}$ be the $(\ell, \ell)$-minor of $M$, and let $D_{ij}$ be the $(-1)^{k+\ell}\cdot (k, \ell)$-minor of $M$, where $\{i, j, k, \ell\}= \{1, 2, 3, 4\}$.
Then, for each dihedral angle $\alpha_{ij}$,
\begin{equation}\label{E:trig-aij}
\cos \alpha_{ij} = \frac{D_{ij}}{\sqrt{D_{ijk}D_{ij\ell}}} \quad\text{and}\quad \sin \alpha_{ij} = \frac{e_{ij}\sqrt{2D}}{\sqrt{D_{ijk}D_{ij\ell}}}.
\end{equation}
Moreover, $D=288V^2$ where $V$ is the volume of the tetrahedron with edge lengths $(e_{ij})_{ij}$.
\end{proposition}
Let $z_{ij} := \exp(\textrm{i} \alpha_{ij})$. Apply $t \mapsto \exp(\textrm{i}t)$ to both sides of Lemma~\ref{lem:case5-integral} to get 
\begin{equation}
\label{E:zeta}
     \prod_{i<j} z_{ij}^{e_{ij}} = \zeta
\end{equation}
for some root of unity $\zeta$. If $z=e^{\textrm{i} \theta}$, then $z^2 - 2 (\cos \theta) z + 1 = 0$.
By \eqref{E:trig-aij}, $\cos \alpha_{ij} \in \pm \sqrt{\Q_{\ge 0}}$.
Let $v \colon \Qbar_p \to \Q \union \{\infty\}$ be the $p$-adic valuation, normalized by $v(p)=1$.
Note, if $\zeta$ is a root of unity, then $v(\zeta)=0$. Apply $v$ to both sides of \eqref{E:zeta} to get 
\begin{equation}
\label{E:valuation combination}
      \sum_{i<j} e_{ij} \, v(z_{ij}) = 0.
\end{equation}

By the theory of Newton polygons,
\[
   v(z_{ij}) = \begin{cases}
       0, &\textup{ if $v(2 \cos \alpha_{ij}) \ge 0$,} \\
       \pm v(2 \cos \alpha_{ij}),   &\textup{ if $v(2 \cos \alpha_{ij}) < 0$.} \\
   \end{cases}
\]
The value of $v(2\cos \alpha_{ij})$ is locally $p$-adically constant, at least in regions where $2 \cos \alpha_{ij}$ does not approach $0$.
Let $K \colonequals \Q(\sqrt{-2D}) = \Q( \textrm{i} V)$. 
By the double angle formula and \eqref{E:trig-aij}
\begin{equation}\label{E:z_ij^2}
    z_{ij}^2 = \cos (2\alpha_{ij})+\mathrm{i}\sin(2\alpha_{ij}) = \frac{2D_{ij}^2 - D_{ijk}D_{ijl}+2 e_{ij} D_{ij}\sqrt{-2D}}{D_{ijk} D_{ijl}}\in K.
\end{equation}

If $-2D$ is a square in $\Q_p$, then there are two ways to embed $K$ into $\Q_p$ 
(mapping $\sqrt{-2D}$ to either possible square root in $\Q_p$), 
and after choosing the embedding $K \injects \Q_p$, the $p$-adic valuation on $\Q_p$
restricts to give a $p$-adic valuation $v$ on $K$.
If $-2D$ is not a square in $\Q_p$, there is a unique extension of the $p$-adic valuation,
but $v(z_{ij}^2)=0$ for all $z_{ij}$, so \eqref{E:valuation combination} is trivial.

As a remark, we note that we lose nothing by passing to $p$-adic valuations: 
by Kronecker's Theorem, if \eqref{E:valuation combination} hold for all $p$-adic valuations $v$, then $\prod z_{ij}^{e_{ij}}$ is a root of unity. This means that $\sum e_{ij} \alpha_{ij} \in \Q\pi$, so the Dehn invariant is actually zero.

\subsection{The new approach}
This approach is suggested by Philipp Habegger and Gabriel Dill, whom we thank especially.
For simplicity, re-index the six edges simply by $\lbrace 1, 2, \dots, 6\rbrace$. Let $n_i\in \mathbb{N}$ be the corresponding rescaled edge lengths from Lemma~\ref{lem:case5-integral}, so that $z_1^{n_1} \cdots z_6^{n_6} = 1$.
However, the $z_i$'s are not rational functions of $n_1,\ldots,n_6$; instead they are algebraic functions.
One can still bound height $H(z_i)$ in terms of $N=\max(n_1,\ldots,n_6) \in \Z_{>0}$.
Arbitrarily order the prime ideals of  $K:=\mathbb{Q}(\sqrt{-2D})$ and let $A$ be the integer matrix whose $ij$-entry is the valuation of $z_i$ at the $j$-th prime ideal, so the $i$-th row of $A$ represents the ``factorization'' of $z_i$. Let $\vec{n} = (n_1, \dots, n_6)$, so $\vec{n}A=0$. As the $z_i$'s span a multiplicative group of rank $5$, $A$ has rank $5$, so we can take a $6 \times 6$ submatrix $B$ of $A$ which has rank $5$. 


As we shall see, some column of adjugate $B^*$ must be a nonzero multiple of $\vec{n}$, so $B^*$ contains an entry of absolute value at least $N$. On the other hand, the entries of $B$ are of absolute value $O(\log N)$, so the entries of $B^*$ are at most $O((\log N)^5)$. Hence, $N$ must be bounded. To make this approach precise, we make use of the following bound on determinants.

\begin{theorem}[Hadamard's Inequality \cite{hadamard1893resolution}]\label{hadamard}
If all entries of $A \in \mathbb{R}^{n\times n}$ have absolute value at most $a$, then $\lvert \det A \rvert \le (\sqrt{n} a)^n$. Moreover, let the rows or columns of $A$ be $\vec{v_1}, \dots \vec{v_n}$ and $\Vert \cdot\Vert_p$ be the $\ell_p$ norm on $\mathbb{R}^{n}$. Then
\[\left| \det A\right| \le \prod_{i=1}^{n} \Vert \vec{v_i} \Vert_{2}  \le \prod_{i=1}^{n} \Vert \vec{v_i} \Vert_{1}.\]
\end{theorem} 
\begin{proof}[Proof of Theorem~\ref{main-case5}]
By \eqref{E:z_ij^2}, $z_i^2 = \alpha_i/b_i$ for some $\alpha_i \in \Z[\sqrt{-2D}]$ and $b_i \in \Z$, so $\Norm(\alpha_i) = \Norm(b_i) = b_i^2$.
Define $\vec{n}$, $N$, $A$, and $B$ as before. Note that $\ker B^\intercal$ is spanned by $\vec{n}$. We claim that some column in the adjugate $B^*$ is a nonzero multiple of $\vec{n}$. As $B B^*=(\det B) I = 0$, any column of the $B^*$ spans $\ker B^\intercal$, so it is a multiple of $\vec{n}$. Since $B$ has rank $5$, at least one of its $5 \times 5$ minors is nonzero, so $B^*$ is not the zero matrix. Then, one of its columns is a nonzero multiple of $\vec{n}$, so some $\left|(B^*)_{k\ell}\right|\ge N$. By Theorem~\ref{hadamard}, for the $i$-th row $\vec{x}_i$ of $B$
\begin{equation}\label{E:N-bd}
N \le \left|(B^*)_{k\ell}\right| = |\det(\text{minor}_{k, \ell}(B))| \le \prod_{i=1}^{5} \|\vec{x}_i\|_1.
\end{equation}

For nonzero prime ideal $\pp$ of ring of integers $\calO$ of $K$, $\pp \intersect \Z$ is a prime ideal $(p)$ of $\Z$. 
The primes for which $v_p$ is nonzero must be splitting primes, and so the exponent of $\pp$ in the factorization of $w_i$ has absolute value bounded by $\max(v_p(b_i^2)), v_p(b_i)) = 2v_p(b_i)$. 
By \cite[Lemma 5]{wirth2014relations}, $|D_{ijk}| \le 3 N^4$, so $|b_i| \le 3^2 N^8$. Hence, for each row $i$ corresponding to prime $p_i$,
\[ \|\vec{x}_i\|_1\le \log_{p_i}(3^2N^8).\]
The right hand side of \eqref{E:N-bd} is maximized when the five primes $p_i$ associated to the rows of the minor of $B$ are $\{2,3,5,7,11\}$. Therefore
\[N \le \prod_{p \in \{2,3,5,7,11
\}} \ \ 2\log_{p}(3^2N^8) = \frac{32\log(3^2N^8)^5}{\log(2)\cdots \log(11)}.\]
A computation yields that maximum edge length $N\le 3.946 \times 10^{12}$, as desired.
\end{proof}
This shows that there are finitely many tetrahedra in Case 5, although it is not feasible to brute force this search space to find them.

\section{Known Dehn Invariant Zero Tetrahedra}
\label{DI0}

In this section, we discuss known Dehn invariant zero families, as well as sporadic cases. 

\subsection{Tetrahedral symmetries}
An obvious symmetry of a tetrahedron arises by permuting the vertex, i.e. the group action of $S_4$ on the labels $\{1,2,3,4\}$ gives rise to ``new'' tetrahedra with the same set of edge lengths and dihedral angles. This gives orbits of size at most $24$.

We also have the \emph{Regge symmetries}. The Regge symmetry fixing the first pair is defined via the following theorem with edges and dihedral angles listed in our conventional order. We can analogously obtain Regge symmetries fixing the second or last pairs. 

\begin{theorem}[Regge Symmetry \cite{akopyan2019regge}]\label{regge}
For a tetrahedron $T$ with edge lengths $(x, y, a, b, c, d)$, there is a tetrahedron $\bar{T}$ of the same volume with edge lengths $(x, y, s-a, s-b, s-c, s-d)$ where $s=(a+b+c+d)/2$. The map $T\mapsto \bar{T}$ is the \emph{Regge symmetry fixing the first pair}.

Moreover, the dihedral angles are mapped analogously: if the dihedral angles of $T$ are $(\theta, \phi, \alpha, \beta, \gamma, \delta)$, then those of $\bar{T}$ are $(\theta, \phi, \sigma-\alpha, \sigma-\beta, \sigma-\gamma, \sigma-\delta)$ where $\sigma = (\alpha+\beta+\gamma+\delta)/2$.
\end{theorem}
\begin{corollary}\label{regge cor}
The set of three sums of opposite edge lengths and the set of three sum of three sums of opposite dihedral angles are invariant under Regge symmetries.
\end{corollary}
\begin{proof}
By symmetry, we check the Regge symmetry fixing the first pair. For $s=({a+b+c+d})/{2}$, we have $(s-a)+(s-b)=2s-a-b = c+d$ and $(s-c)+(s-d)=2s-c-d=a+b$. Analogous statements hold for other Regge symmetries and for dihedral angles.
\end{proof}

\subsection{The one-parameter families}
The authors are aware of four distinct infinite families of Dehn invariant zero tetrahedra up to Regge symmetries. Three were known previously as the Hill Families $H_{1}, H_{2}$, and $H_{3}$, which are discovered in \cite{MR1576480} and listed in \cite{boltianskiui1979hilbert}. The following lemma shows that they are still one-parameter families when restricted to tetrahedra with integer edge lengths. The
proof is deferred to Section~\ref{valid-sec}.

\begin{lemma}\label{hill-valid}
Each Hill family contains infinitely many tetrahedra with integer edge lengths. 
\end{lemma}

The fourth ``new'' family has been recently discovered by the authors. In the table below, we list all of them them up to Regge and $S_4$ symmetries. We also compute the dimension of the span of the dihedral angles over $\Q$. We let real numbers $t>4$ and $\alpha \in (0, \pi)$ be the free parameters of each one-parameter family.


{\scriptsize \begin{table}[h]
\centering
\caption{All known one-parameter families up to Regge and $S_4$ symmetries}
\begin{tabular}{ |c|c|c|c| }
\hline
Family &Dihedral Angles &  Edge Lengths & Dimension\\
\hline
$H_{1}$ & $(\alpha, \alpha, \pi/3, \pi-2\alpha, \pi/2, \pi/2)$ & $(\sin\alpha, \sin\alpha, \sqrt{3}\cos\alpha, \sin\alpha, 1, 1)$ & $1$ \\
\hline
$H_{2}$  & $(\alpha, \beta, \pi/3, \pi/2-\alpha, \pi/2, \pi-\beta) $& $\left(2\sin\alpha, \sqrt{5\sin^2\alpha - 1}, \sqrt{3}\cos\alpha, 2\sin\alpha, 2, \sqrt{5\sin^2\alpha-1}\right)$ & $2$\\
\hline 
$H_{3}$  & $(\alpha, \gamma, \pi/6, \pi-2\alpha, \pi-\gamma, \pi/2) $ & $\left(2\sin\alpha, \sqrt{\sin^2\alpha +2}, \sqrt{12}\cos\alpha, \sin\alpha, \sqrt{\sin^2\alpha+2}, 2\right)$ & $2$ \\ 
\hline 
New & See Table 2 & $(t+1, t-1, t+1, t-1, t, 6)$ & $2$
\\ 
\hline
\end{tabular}
\end{table}}

\begin{remark}
The last column describes the generic dimension of the family.    
\end{remark}
In the following table, we list out each Regge orbit, thereby obtaining a complete list of known families up $S_4$ symmetry.

\begin{landscape}

{\tiny

\begin{table}[h]
    \centering
    \caption{All known one-parameter families up to $S_{4}$ symmetry}
    \label{up to S4} 
    \begin{tabular}{|c|c|c|}
    \hline   
Family & Dihedral Angles &  Edge Lengths \\ \hline

\m{3}{$H_{1}$}  &         $\left( \alpha,   \alpha,  \frac{\pi}{3},  \pi - 2 \alpha,  \frac{\pi}{2},  \frac{\pi}{2}\right)$  & $\left( s, s, s,  \sqrt{3} c, 1,   1\right)$ \\ \cline{2-3} & 
         $\left( - \alpha + \frac{2 \pi}{3}, - \alpha + \frac{2 \pi}{3}, \alpha + \frac{\pi}{6}, - \alpha + \frac{5 \pi}{6},   \alpha,  \alpha\right)$ & $\left( s,   s,  \frac{s-\sqrt{3} c}{2} + 1,   \frac{s-\sqrt{3} c}{2} + 1,   \frac{s+\sqrt{3} c}{2},   \frac{s+\sqrt{3} c}{2}\right)$ \\ \cline{2-3} & 
          $\left( - \alpha + \frac{2 \pi}{3}, - \alpha + \frac{2 \pi}{3}, 2 \alpha - \frac{\pi}{3}, \frac{\pi}{3}, \frac{\pi}{2}, \frac{\pi}{2}\right)$ & $\left( \frac{s+\sqrt{3} c}{2}, \frac{s+\sqrt{3} c}{2}, \frac{3 s-\sqrt{3} c}{2}, \frac{s+\sqrt{3} c}{2}, 1, 1\right)$
       \\  \hline

\m{6}{$H_{2}$} & $\left( \beta,  \alpha,  \frac{\alpha}{2} - \frac{\beta}{2} + \frac{2 \pi}{3},  - \frac{\alpha}{2} - \frac{\beta}{2} + \frac{5 \pi}{6},  - \frac{\alpha}{2} - \frac{\beta}{2} + \frac{2 \pi}{3},  - \frac{\alpha}{2} + \frac{\beta}{2} + \frac{\pi}{6} \right)$ 
        &  
$\left( 2s,  f,  \dfrac{f}{2} + s - \frac{\sqrt{3} c}{2} + 1, \frac{f}{2} - s + \frac{\sqrt{3} c}{2} + 1 ,\right. \quad \left. \frac{f}{2} + s + \frac{\sqrt{3} c}{2} - 1,  - \frac{f}{2} + s + \frac{\sqrt{3} c}{2} + 1\right)$  \\ \cline{2-3} & 
        
        $\left( \frac{\alpha}{2} - \beta + \frac{3 \pi}{4}, - \frac{\alpha}{2} + \frac{3 \pi}{4}, \alpha + \frac{\beta}{2} - \frac{\pi}{12}, \frac{\beta}{2} + \frac{\pi}{12}, - \frac{\alpha}{2} + \frac{\beta}{2} + \frac{\pi}{6}, - \frac{\alpha}{2} - \frac{\beta}{2} + \frac{2 \pi}{3}\right)$ 
        
        & 
        
        $ \left( f - s + 1, s + 1, \frac{f}{2} + 2 s - \frac{\sqrt{3} c}{2},\frac{f}{2} + \frac{\sqrt{3} c}{2}, - \frac{f}{2} + s + \frac{\sqrt{3} c}{2} + 1, \frac{f}{2} + s + \frac{\sqrt{3} c}{2} - 1\right)$ 

        \\ \cline{2-3} &

        $\left( - \frac{\alpha}{2} + \frac{3 \pi}{4}, \frac{\alpha}{2} - \beta + \frac{3 \pi}{4}, - \alpha + \frac{\pi}{2}, \frac{\pi}{3}, \frac{\alpha}{2} + \beta - \frac{\pi}{4}, \frac{\alpha}{2} + \frac{\pi}{4}\right)$

        & $\left( s + 1, f - s + 1, 2 s, \sqrt{3} c, f + s - 1, s + 1\right)$

        \\ \cline{2-3} & 

    $\left( - \frac{\beta}{2} + \frac{5 \pi}{12}, - \alpha + \frac{\beta}{2} + \frac{5 \pi}{12}, \frac{\beta}{2} + \frac{\pi}{12}, \alpha + \frac{\beta}{2} - \frac{\pi}{12}, \pi - \beta, \frac{\pi}{2}\right)$

    &
    
    $\left( \frac{f}{2} + \frac{\sqrt{3} c}{2}, - \frac{f}{2} + 2 s + \frac{\sqrt{3} c}{2}, \frac{f}{2} + \frac{\sqrt{3} c}{2}, \frac{f}{2} + 2 s - \frac{\sqrt{3} c}{2}, f, 2\right)$ 
    
    \\ \cline{2-3} & 

    $\left( \alpha, \beta, - \alpha + \frac{\pi}{2}, \frac{\pi}{3}, \pi - \beta, \frac{\pi}{2}\right) $

    & 

    $\left( 2 s, f, 2 s, \sqrt{3} c, f, 2\right)$

    \\ \cline{2-3} & 

    $\left( - \frac{\beta}{2} + \frac{5 \pi}{12}, - \alpha + \frac{\beta}{2} + \frac{5 \pi}{12}, \frac{\alpha}{2} - \frac{\beta}{2} + \frac{2 \pi}{3}, - \frac{\alpha}{2} - \frac{\beta}{2} + \frac{5 \pi}{6}, \frac{\alpha}{2} + \beta - \frac{\pi}{4}, \frac{\alpha}{2} + \frac{\pi}{4}\right)$
&$\left( - \frac{f}{2} + 2 s + \frac{\sqrt{3} c}{2}, \frac{f}{2} + \frac{\sqrt{3} c}{2}, \frac{f}{2} + s - \frac{\sqrt{3} c}{2} + 1,\frac{f}{2} - s + \frac{\sqrt{3} c}{2} + 1, f + s - 1, s + 1\right) $
        \\ \hline
\m{6}{$H_{3}$} & $\left( - \frac{3 \alpha}{2} + \frac{\gamma}{2} + \frac{7 \pi}{12}, - \frac{\alpha}{2} - \frac{\gamma}{2} + \frac{7 \pi}{12}, - \frac{\alpha}{2} + \frac{\gamma}{2} + \frac{5 \pi}{12}, \frac{3 \alpha}{2} + \frac{\gamma}{2} - \frac{5 \pi}{12}, \pi - \gamma, \frac{\pi}{2}\right)$ &
$\left( - \frac{r}{2} + \frac{3 s}{2} + \sqrt{3} c, \frac{r}{2} - \frac{s}{2} + \sqrt{3} c, \frac{r}{2} + \frac{3 s}{2} - \sqrt{3} c, \frac{r}{2} + \frac{s}{2} + \sqrt{3} c, 2, r\right)$

\\ \cline{2-3} &

$\left( \frac{\alpha}{2} - \gamma + \frac{3 \pi}{4}, - \frac{\alpha}{2} + \frac{3 \pi}{4}, \pi - 2 \alpha, \frac{\pi}{6}, \frac{\alpha}{2} + \gamma - \frac{\pi}{4}, \frac{\alpha}{2} + \frac{\pi}{4}\right)$&$\left( s + 1, r - s + 1, s, 2 \sqrt{3} c, s + 1, r + s - 1\right)$

\\ \cline{2-3} &

$\left( \frac{\alpha}{2} - \gamma + \frac{3 \pi}{4}, - \frac{\alpha}{2} + \frac{3 \pi}{4}, \frac{3 \alpha}{2} + \frac{\gamma}{2} - \frac{5 \pi}{12}, - \frac{\alpha}{2} + \frac{\gamma}{2} + \frac{5 \pi}{12}, - \alpha - \frac{\gamma}{2} + \frac{5 \pi}{6}, - \alpha + \frac{\gamma}{2} + \frac{\pi}{3}\right)
$&$
\left( r - s + 1, s + 1, \frac{r}{2} + \frac{3 s}{2} - \sqrt{3} c, \frac{r}{2} + \frac{s}{2} + \sqrt{3} c, \frac{r}{2} + \frac{s}{2} + \sqrt{3} c - 1, - \frac{r}{2} + \frac{s}{2} + \sqrt{3} c + 1\right)
$

\\ \cline{2-3} &

$\left( \alpha, \gamma, \alpha - \frac{\gamma}{2} + \frac{\pi}{3}, - \alpha - \frac{\gamma}{2} + \frac{7 \pi}{6}, - \alpha - \frac{\gamma}{2} + \frac{5 \pi}{6}, - \alpha + \frac{\gamma}{2} + \frac{\pi}{3}\right)$&

$\left( r, 2 s, \frac{r}{2} + \frac{s}{2} - \sqrt{3} c + 1, \frac{r}{2} - \frac{s}{2} + \sqrt{3} c + 1, \frac{r}{2} + \frac{s}{2} + \sqrt{3} c - 1, - \frac{r}{2} + \frac{s}{2} + \sqrt{3} c + 1\right)
$

\\ \cline{2-3} &

$\left( \alpha, \gamma, \frac{\pi}{6}, \pi - 2 \alpha, \pi - \gamma, \frac{\pi}{2}\right)$&$\left( 2 s, r, 2 \sqrt{3} c, s, r, 2\right)$

\\ \cline{2-3} &

$\left( - \frac{\alpha}{2} - \frac{\gamma}{2} + \frac{7 \pi}{12}, - \frac{3 \alpha}{2} + \frac{\gamma}{2} + \frac{7 \pi}{12}, \alpha - \frac{\gamma}{2} + \frac{\pi}{3},  - \alpha - \frac{\gamma}{2} + \frac{7 \pi}{6}, \frac{\alpha}{2} + \frac{\pi}{4}, \frac{\alpha}{2} + \gamma - \frac{\pi}{4}\right)  
 $&$\left( \frac{r}{2} - \frac{s}{2} + \sqrt{3} c, - \frac{r}{2} + \frac{3 s}{2} + \sqrt{3} c, \frac{r}{2} - \frac{s}{2} + \sqrt{3} c + 1, \frac{r}{2} + \frac{s}{2} - \sqrt{3} c + 1, s + 1, r + s - 1\right)$
   \\ \hline
\m{3}{New}&$\begin{array}{l}
              \left(\arccos \frac{\sqrt{3(t+4)}}{(t-2)\sqrt{t+2}}, \arccos \frac{\sqrt{3(t-4)}}{(t+2)\sqrt{t-2}}, \arccos \frac{\sqrt{3(t+4)}}{(t-2)\sqrt{t+2}}, \right. \\
              \quad \quad  
              \left. \arccos \frac{\sqrt{3(t-4)}}{(t+2)\sqrt{t-2}}, \arccos \frac{t^{2}-28}{(t-2)(t+2)}, \arccos \frac{\sqrt{(t-4)(t+4)}}{2\sqrt{(t-2)(t+2)}}\right) 
         \end{array}$ & $(t+1, t-1, t+1, t-1, t, 6)$ \\ \cline{2-3}
&$\begin{array}{cc}
             \left( \arccos \frac{\sqrt{3(t+4)}}{(t-2)\sqrt{t+2}}, \arccos \frac{\sqrt{3(t-4)}}{(t+2)\sqrt{t-2}}, s-\arccos \frac{\sqrt{3(t+4)}}{(t-2)\sqrt{t+2}}, \right. \\ 
            \quad \quad \left. u-\arccos \frac{\sqrt{3(t-4)}}{(t+2)\sqrt{t-2}}, u-\arccos \frac{t^{2}-28}{(t-2)(t+2)}, u-\arccos \frac{\sqrt{(t-4)(t+4)}}{2\sqrt{(t-2)(t+2)}}\right)
        \end{array}$ & $\left(t+1, t-1,\frac{t}{2}+2, \frac{t}{2}+4, \frac{t}{2}+3, \frac{3t}{2}-3\right)$ \\ \hline
    \end{tabular}
\end{table}}
\vspace{-1 ex}
where we use following variables that are functions of free parameters $t$ and/or $\alpha$:
{\small \begin{itemize}
\item $s=\sin \alpha$,
\item $c=\cos \alpha $,
\item $\beta = \arccos  \left(\frac{1}{2}\cot \alpha\right)$,
\item $\gamma = \arccos \left(\frac{1}{\sqrt{3}}\cos\alpha\right)$,
\item $r=\sqrt{\sin^{2}{\alpha} + 2}$,
\item $f=\sqrt{5 \sin^{2}{\alpha} - 1}$,
\item $u = \frac{1}{2} \left[\arccos \left(\frac{\sqrt{3(t+4)}}{(t-2)\sqrt{t+2}}\right)+ \arccos \left(\frac{\sqrt{3(t-4)}}{(t+2)\sqrt{t-2}}\right)+
\arccos \left(\frac{t^{2}-28}{(t-2)(t+2)}\right)+
\arccos \left(\frac{\sqrt{(t-4)(t+4)}}{2\sqrt{(t-2)(t+2)}}\right)\right]$.
\end{itemize}}
\end{landscape}

\subsection{Isolated examples}
Since \cite{kedlaya2020space} follows the same notation, we redirect the reader there for the $59$ sporadic rational tetrahedra found in the appendix. These tetrahedra all have trivial span of dihedral angles over $\Q$, since all dihedral angles are rational multiples of $\pi$.
We list the examples in \cite{boltianskiui1979hilbert} not in \cite{kedlaya2020space} into the same form. They are hence all of dimension at least $1$. We group them by Regge equivalence and compute dimensions.
{\scriptsize \begin{table}[h]
\centering
\caption{Known isolated examples of Dehn invariant zero tetrahedra in \cite{boltianskiui1979hilbert} but not in \cite{kedlaya2020space}}
\begin{tabular}{ |c|c|c| }
\hline
Edge Lengths & Dihedral Angles & $\dim$ \\
\hline
$(\sqrt[4]{5}\sqrt{\tau}, \sqrt{7}/2, \sqrt{3}, \sqrt{3}/2, \sqrt{7}/2, \sqrt[4]{5}/\sqrt{\tau})$ & $(3\pi/5, \pi - \alpha_{1}, \pi/6, \pi/3, \alpha_{1}, \pi/5)$ & 1 \\
\hline
$(\sqrt[4]{5}\sqrt{\tau}, \sqrt{7+3\tau}/2, \sqrt{3}, \sqrt{7+3\tau}/2, \sqrt[4]{5}/\sqrt{\tau})$ & $(3\pi/5, \pi - \alpha_{2}, \pi/3, \alpha_{2}, \pi/5, \pi/10)$ & 1 \\
$(\sqrt{2+\tau}, 2\sqrt{18 - 11\tau}, 2\sqrt{6 - 3\tau}, \sqrt{7+3\tau}/\tau^2, 2\sqrt{5} - 2, \sqrt{7+3\tau}/\tau^2)$ & $(\pi/5, 3\pi/10, \pi/3, \alpha_{2}, \pi/2, \pi - \alpha_{2})$ & 1 \\
\hline 
$(\sqrt[4]{5}\sqrt{\tau}, \sqrt[4]{5}\sqrt{\tau/2}, \sqrt{3}, \sqrt{7-3/\tau}/2, \sqrt{7-3/\tau}/2, \sqrt[4]{5}/\sqrt{\tau})$ & $(3\pi/10, 3\pi/5, \pi/3, \pi - \alpha_{3}, \alpha_{3}, \pi/5)$ & 1 \\
$(2\sqrt{2+\tau}, \sqrt{18 - 11\tau}, 2\sqrt{6 - 3\tau}, \sqrt{10 - 3\tau}, \sqrt{10 - 3\tau}, 2\sqrt{5} - 2)$ & $(\pi/10, 3\pi/5, \pi/3, \alpha_{3}, \pi - \alpha_{3}, \pi/2)$ & 1 \\
\hline 
$(\sqrt[4]{5}\sqrt{\tau}, 1, \sqrt{3}, 1, 1, \sqrt[4]{5}/\sqrt{\tau})$ & $(3\pi/10, \alpha_{6}, \pi/6, \alpha_{5}, \alpha_{4}, \pi/10)$ & 2 \\
\hline 
$(\sqrt{3}, \sqrt{5}/2, \sqrt{3}, \sqrt{5}/2, \sqrt{5}/2, 2)$ & $(\pi/6, \pi - \alpha_{7}/2, \pi/6, \pi - \alpha_{7}/2, \alpha_{7}, \pi/4)$ & 1 \\

\hline 
$(2\sqrt{3}(\tau - 1), \sqrt{2+\tau}, 2\sqrt{3 - \tau}, \sqrt{10 - 3\tau}, \sqrt{10 - 3\tau}, 2\sqrt{3})$ & $(\pi/6, \pi/5, \pi/5, \alpha_{8}, \pi - \alpha_{8}, 2\pi/3)$ & 1 \\
\hline 
$(2\sqrt{2+\tau}, \sqrt{3}(\tau - 1), 2\sqrt{3 - \tau}, \sqrt{6 + \tau}, \sqrt{6+\tau}, 2\sqrt{3})$ & $(\pi/10, \pi/3, \pi/5, \alpha_{9}, \pi - \alpha_{9}, 2\pi/3)$ & 1 \\
\hline 
$(\tau\sqrt{3}, 2\sqrt{3 - \tau}, 2\sqrt{3}, \sqrt{7-\tau}, 2\sqrt{2+\tau}, \sqrt{7-\tau})$ & $(\pi/3, 3\pi/10, \pi/3, \alpha_{10}, 2\pi/5, \pi - \alpha_{10})$ & 1 \\
\hline 
$(\sqrt{3}, \sqrt{7-4\tau}/2, \sqrt{6-3\tau}, \sqrt{11-4\tau}/2, \sqrt{11-4\tau}/2, \sqrt{3-\tau})$ & $(\pi/6, 3\pi/5, \pi/3, \alpha_{11}, \pi - \alpha_{11}, 2\pi/5)$ &  1\\
\hline 
$(\sqrt{6-3\tau}, \sqrt{7-4\tau}, \sqrt{3}, \sqrt{6-3\tau}, 2\sqrt{4-2\tau}, \sqrt{3-\tau})$ & $(\alpha_{12}, \pi/5, \pi/3, 2\pi/3 - \alpha_{12}, \pi/2, 3\pi/5)$ & 1 \\
$(\sqrt{3}, \sqrt{3}(\tau - 1), \sqrt{3}, \sqrt{2+\tau}, \sqrt{3-\tau}, 2\sqrt{2})$ & $(\pi/3 - \alpha_{12}, \pi/3, \alpha_{12}, 2\pi/5, 4\pi/5, \pi/2)$ & 1 \\
\hline
\end{tabular}
\end{table}}

where length and angle constants are
{\small
\begin{multicols}{2}
\begin{itemize}
    \item $\tau=(\sqrt{5}+1)/2$,
    \item $u=\tau^2/(2\sqrt{3})$,
    \item $v=\tau^{3/2}/(2\sqrt[4]{5})$,
    \item $w=\sqrt{11+16\tau}$,
    \item $z=1/(2\sqrt{3}\tau^2)$,
    \item $\alpha_1=\arctan\sqrt{7/5}$,
    \item $\alpha_2=\arctan\sqrt{9-2\sqrt{5}}$,
    \item $\alpha_3=\arctan\sqrt{9+2\sqrt{5}}$,
    \item $\alpha_4=\pi - \arctan\sqrt{9+2\sqrt{5}}$,
    \item $\alpha_5=\pi - \arctan 2$,
    \item $\alpha_6=2\pi - \alpha_4-\alpha_5$,
    \item $\alpha_7=\pi - \arccos \left(2/3\right)$,
    \item $\alpha_8= \arccos u$,
    \item $\alpha_9= \arccos v$,
    \item $\alpha_{10} = \arctan w$,
    \item $\alpha_{11} = \arccos z$,
    \item $\alpha_{12}= \arctan\sqrt{3/5}$.
\end{itemize}
\end{multicols}}


\section{Properties of the One-parameter Families}
In this section, we show Lemma~\ref{hill-valid} and Theorem~\ref{new-fam} about the one-parameter families.
\subsection{Validity of the one-parameter families}\label{valid-sec}
We show that the three Hill families and the two new families $T$ and $T'$ from \eqref{E:TT'} are one-parameter families of Dehn invariant zero tetrahedra with integer edge lengths. Namely, we prove Lemma~\ref{hill-valid} and Theorem~\ref{new-fam}(1).
\begin{lemma}\label{H1-valid}
$H_{1}$ contains infinitely many tetrahedra with integer edge lengths. 
\end{lemma}
\begin{proof}
It suffices to find infinitely many $\alpha$ 
such that $\sin \alpha$ and $\sqrt{3} \cos \alpha$ are both rational.
Let $x = \sin \alpha$, so we want rational solutions to
$3(1-x^2) = y^2$. Other than $(-1, 0)$, they are
\[
   (x, y)= \left(\frac{3-t^2}{3+t^2}, \frac{6t}{3+t^2}\right)
\]
for $t\in \Q$. We see that $(\sin \alpha, \sqrt{3}\cos\alpha) = (x, y)\in \Q^2$ holds for infinitely many $\alpha$.
\end{proof}
\begin{lemma}\label{H2-valid}
$H_{2}$ contains infinitely many tetrahedra with integer edge lengths.
\end{lemma}
\begin{proof}
It suffices to find infinitely many $\alpha$ 
such that $2 \sin \alpha$, $\sqrt{3} \cos \alpha$, and $\sqrt{5 \sin^2 \alpha -1}$ are all rational.
Let $x = \sin \alpha$, so we want rational solutions to $3(1-x^2) = y^2$ and $5x^2-1 = z^2$.
%
As in Lemma~\ref{H1-valid}, than $(-1, 0)$, the rational solutions $(x, y)$ to the first equation are
\[ (x, y) = \left(\frac{3-t^2}{3+t^2}, \frac{6t}{3+t^2}\right)\]
for $t\in \Q$. Then, the second equation says $5 \left( \frac{3-t^2}{3+t^2} \right)^2 - 1$ is a square, i.e.
\[
  5 (3-t^2)^2 - (3+t^2)^2 = (2w)^2 \implies w^2 = t^4 - 9 t^2 + 9
\]
for some $w\in \Q$. Using MAGMA, this algebraic curve is birational to the elliptic curve
\[
    E: y^2 = x^3 - 9 x^2 - 36 x + 324.
\]
Moreover, $E$ is an elliptic curve of rank $1$,
with $E(\Q) = \Z \times C_2 \times C_2$, so it has infinitely many rational solutions. Therefore, $(\sin \alpha, \sqrt{3}\cos\alpha, \sqrt{5\sin^2\alpha -1})\in \Q^3$ for infinitely many $\alpha$.
\end{proof}
An analogous proof via birationality to an elliptic curve of positive rank shows the following. 
\begin{lemma}\label{H3-valid}
$H_{3}$ contains infinitely many tetrahedra with integer edge lengths.
\end{lemma}
The three lemmas above show Lemma~\ref{hill-valid}. 
We show Theorem~\ref{new-fam}(1) for the new family of tetrahedra $T$ via Propositions~\ref{new-lem-1} and~\ref{new-lem-2}. Then, the same holds for $T'$ by a Regge symmetry.
\begin{proposition}\label{new-lem-1}
For every $t>4$, $T$ is a non-degenerate tetrahedron.
\end{proposition}
\begin{proof}
By \cite[Lemma 4]{wirth2014relations}, suffices to check the triangle inequality for each face and that the Cayley-Menger determinant $D$ is positive. Both are true for $t>4$ since $D = 216t^2(t-4)(t+4)$ and the faces have side lengths $(t+1, t+1, 6), (t-1, t-1, 6), (t-1, t, t+1)$.
\end{proof}
\begin{proposition}\label{new-lem-2}
For every $t>4$, $T$ has Dehn invariant zero and its dihedral angles satisfy 
\begin{equation}\label{E:int-rels}
2\alpha_{12}+2\alpha_{24}+\alpha_{14} = 2\pi\quad\text{and}\quad \alpha_{12}-\alpha_{24}+3\alpha_{23} = \pi.
\end{equation}
\end{proposition}
\begin{proof}
By Proposition~\ref{prop-WD}, we compute the minors of the Cayley-Menger and cosines
{\small \[\left(\cos \alpha_{12}, \cos \alpha_{24}, \cos \alpha_{23}\right) = \left(\frac{\sqrt{3(t+4)}}{(t-2)\sqrt{t+2}}, \frac{\sqrt{3(t-4)}}{(t+2)\sqrt{t-2}}, \frac{t^{2}-28}{(t-2)(t+2)}\right). \]}
Using trigonometric identities, we compute that
{\small \begin{align*}
\cos (2\alpha_{12}+2\alpha_{24}) & = 2\left(\cos \alpha_{12}\cos\alpha_{24}-\sin\alpha_{12}\sin\alpha_{24}\right)^2-1 = \frac{t^{2}-28}{(t-2)(t+2)} ,\\
\cos (\alpha_{12}-\alpha_{24}) & = \cos\alpha_{12}\cos\alpha_{24}-\sin\alpha_{12}\sin\alpha_{24}
 = \frac{(t^{2}+2)\sqrt{(t-4)(t+4)}}{\sqrt{(t-2)(t+2)}},\\
\cos (3\alpha_{23}) & = 4\cos^{3}\alpha_{23}-3\cos\alpha_{23} = -\frac{(t^{2}+2)\sqrt{(t-4)(t+4)}}{\sqrt{(t-2)(t+2)}}.
\end{align*}}

Hence, $\cos \alpha_{23}=\cos(2\alpha_{12}+2\alpha_{24}) $ and
$\cos (3\alpha_{23})=-\cos(\alpha_{12}-\alpha_{24})$,
so \eqref{E:int-rels} holds. 
Then, $T$ Dehn invariant zero since
\[
\sum_{i<j}e_{ij}\alpha_{ij} =  2(t+1)\alpha_{12}+2(t-1)\alpha_{24}
+t\alpha_{14}+6\alpha_{23} =2(t+1)\pi.\qedhere
\]
\end{proof}
\begin{corollary}\label{dim-le-2}
The dimension of the $\Q$-span of the dihedral angles of $T$ is at most $2$. 
\end{corollary} 
\begin{proof}
The two relations above and symmetries $\alpha_{12}=\alpha_{13}$ and $\alpha_{24}=\alpha_{34}$ give $4$ relations.
\end{proof}
\subsection{Dimension of span of angles of the new families}
We improve Corollary~\ref{dim-le-2}: we show that for both new families the dimension is almost always $2$, proving Theorem~\ref{new-fam}(2).
\begin{proposition}
For both new families $T$ and $T'$, the dimension of the $\Q$-span of the dihedral angles of the new family is exactly $2$ almost everywhere for $t>4$.
\end{proposition}
\begin{proof}
For $T$, form the 6-tuple $(z_{ij}) = (\exp(\textup{i}\alpha_{ij}))$ as in Section~\ref{sec-val}, these lie on an algebraic curve $C$ in $(\C^\times)^6$.
Each system of possible multiplicative relations defines a subgroup of $(\C^\times)^6$.
We know that $C$ lies in a $2$-dimensional subgroup $H$ of $(\C^\times)^6$
defined by certain specific products of powers of the $z_{ij}$ being $1$.
Here $H \isom (\C^\times)^2$.
Then, using the coordinates on $H$, the equation of $C$ is a single polynomial equation $f(x,y)=0$
in two variables $x$ and $y$.

To determine the dimension of the span, we will try specific integers $t$.  Compute the $z_{ij}^2$ as elements of $\Q(\sqrt{-2D})$.
Let $M$ be the multiplicative subgroup of $\C^\times$ generated by the $z_{ij}^2$.
We know that $M$ has rank at most $2$. We take the case of $t=16$ and compute
\[ \cos \alpha_{12} = \frac{\sqrt{5}}{7\sqrt{6}}\quad\text{and}\quad 
\cos \alpha_{34} = \frac{1}{3\sqrt{14}}.\] Now, using $z^{2}=\cos(2\alpha)+\textrm{i}\sin(2\alpha)$, we have
that \[
z_{12}^{2}=\frac{-142+17a}{147}\quad\text{and}\quad 
z_{34}^{2}=\frac{-62+5a}{63},
\]
where $a=\sqrt{-5}$. Take $v_{3}$ with $a\equiv 2\pmod 3$ and compute $a=2+2\cdot 3^{2}+\dots$ which implies
\begin{align*}
v_{3}(z_{12}^{2})&=v_{3}(-142+17a)-v_{3}(147)=2-1=1,\\
v_{3}(z_{34}^{2})&=v_{3}(-62+5a)-v_{3}(63)=0-2=-2.
\end{align*}
Now, take $v_{7}$, pick $a\equiv 4\pmod{7}$ and compute $a=4-3\cdot 7 +7^{2}+\dots$. Then
\begin{align*}
v_{7}(z_{12}^{2})&=v_{7}(-142+17a)-v_{7}(147)=0-1=-1,\\
v_{7}(z_{34}^{2})&=v_{7}(-62+5a)-v_{7}(63)=2-1=1.
\end{align*}
We see that $M$ in this case is invertible: 
\[ M=\begin{pmatrix}
v_{3} (z_{12}^{2}) & v_{3} (z_{34}^{2}) \\
v_{7} (z_{12}^{2}) & v_{7} (z_{34}^{2}) 
\end{pmatrix}
= \begin{pmatrix}
1 & -2 \\
-1 & 1
\end{pmatrix}.
\]
Therefore, the dimension of the span of the dihedral angles over $\mathbb{Q}$ is $2$ in this case. 
There are countably many $\Q$-linear relations among the two angles $\alpha_{12}, \alpha_{34}$, each of which has finitely many solutions. Hence, the dimension could only descend below $2$ countably many times. 

By Theorem~\ref{regge}, the two Regge equivalent families $T$ and $T'$ have the same span of dihedral angles as we can write the two $6$-tuples of dihedral angles as rational linear combinations of each other by Regge symmetry. Hence, $T$ and $T'$ have dimension $2$ almost always $2$.
\end{proof}

\begin{remark}
This shows that the dimension is almost always $2$. Moreover, we can show that it only descends countably many times: to see this, observe that there are countably many $\Q$-linear relations among the two angles $\alpha_{12}, \alpha_{34}$, each of which has finitely many solutions. 
\end{remark}
\subsection{Regge equivalents}
We consider the orbit under Regge symmetries of a tetrahedron $T$ of the first new family. We show Theorem~\ref{new-fam}(3) in the following two propositions.

\begin{proposition}
The new family $T$ is not Regge equivalent to any of the three Hill families.
\end{proposition}
\begin{proof}
By Corollary~\ref{regge cor}, it suffices to check that the set of three sums of opposite dihedral angles of the Hill families are distinct from that of $T$.
In $T$, $\alpha_{12}=\alpha_{13}$ and $\alpha_{34}=\alpha_{24}$, so two of the three sums are equal.
We show that this is not true for any of the Hill families.

For the first family, the sums are $\left\lbrace 2\alpha, {4}/{3}\pi - 2\alpha, \pi\right\rbrace$. If we set any two equal, there is only one value of $\alpha$, so they are generically distinct.
For the second family, the sums are $\left\lbrace \alpha+\beta, {5}/{6}\pi-\alpha, {3}/{2}\pi -\beta\right\rbrace$ with $2\cos\beta = \cot\alpha$. We see that they are generically distinct: for $\alpha = \beta = \pi/6$, $2\cos \beta = \cot \alpha = \sqrt{3}$ and $\left( \alpha+\beta, \frac{5}{6}\pi-\alpha, \frac{3}{2}\pi -\beta\right) = (\pi/3, 2\pi/3, 4\pi/3)$.
\end{proof}

Now, we show the two new families form a Regge equivalence class, proving Theorem~\ref{new-fam}(3).


\begin{proposition}
Fix any real number $t>4$. The Regge orbit of $T$ is $\{T,T'\}$.
\end{proposition}
\begin{proof}
For $i\in \{1, 2, 3\}$, let $R_i$ be the Regge symmetry fixing the $i$-th pair. We check that
\[ R_1(T)=R_2(T)=R_3(T')=T'\quad\text{and}\quad R_1(T')=R_2(T')=R_3(T)=T.\]
 Therefore, 
 the orbit of $T$ is $\lbrace T, T'\rbrace$.
\end{proof}
\begin{appendices}
\section{Computational Aspects}
\label{comp}
We describe computational aspects of this problem. The relevant code is given in \cite{tetrapy}.

\subsection{Dehn invariant zero tetrahedra of integer edge lengths}

It is relatively easy to test whether a given tuple of positive integers is the tuple of edge lengths of a Dehn invariant zero tetrahedron.
Recall from Section~\ref{sec-val} that if all edge lengths of a tetrahedron are integral, then the Dehn invariant is zero if and only if $\zeta$ is a root of unity, where
\[ \zeta := \prod z_{ij}^{e_{ij}} \in \Q(\sqrt{-2D}).\] 
It is easy to classify all roots of unity which lie in a quadratic number field: they are 
\[ \left\{1, -1, \textrm{i}, -\textrm{i} , \frac{-1+\textrm{i}\sqrt{3}}{2}, \frac{-1-\textrm{i}\sqrt{3}}{2}  \right\}.\]
We check this equality to test tetrahedra for Dehn invariant zero. For example, the tetrahedra with edge lengths $(13, 16, 14, 16, 8, 12)$ and $(13, 13, 11, 19, 12, 11)$ have Dehn invariant zero. 

There is a one-way test which could expedite the computation in the case a tetrahedron has Dehn invariant zero. Instead of checking whether $\zeta$ is a root of unity in a quadratic number field, we could reduce this computation modulo some prime $p$, and check the equality in $\mathbb{F}_p$.

\subsection{Determine the dimension of a tetrahedron}

It is also possible to algorithmically determine the dimension of the $\Q$-span of the dihedral angles of a tetrahedron in $\R/\Q\pi$. For simplicity, we restrict ourselves to the case where the edge lengths are integral. 

Given the dihedral angles of a tetrahedron, we first compute the values $z_j^2 = \exp(\textrm{i} \alpha_j)$, which live in $\Q(\sqrt{-2D})$. The dimension of the span is equal to the rank of the multiplicative group generated by the $z_j^2$'s. In order to compute this rank, it suffices to determine the space of $(n_1, \cdots, n_6)$ such that $ \prod z_i^{2n_i}$ is a root of unity, which is equivalent to vanishing valuations at each place by Kronecker's theorem.

We first factor the fractional ideals generated by the $z_j^2$'s. For any prime factor $\mathfrak{p}$ which appears in the factorization of one of the six $z_j^2$, we compute the rows $$\begin{pmatrix}
v_\mathfrak{p}(z_1^2) & \dots & v_\mathfrak{p}(z_6^2)
\end{pmatrix}.$$

Combining these rows gives us a matrix $A$ of integer entries, and the desired dimension is equal to the rank of $A$ as a $\Q$-matrix. For example, the tetrahedra $(13, 16, 14, 16, 8, 12)$ and $(13, 13, 11, 19, 12, 11)$ from the previous example have dimension two. 
\end{appendices}
\bibliographystyle{alpha}
\bibliography{ref}

\end{document}